\documentclass[11pt]{article}%
\usepackage[letterpaper, portrait, margin=1in, headheight=0pt]{geometry}
\usepackage{mathtools}
\usepackage{enumerate}
\usepackage{amsmath,amsthm,amssymb,amsfonts}
\usepackage{enumitem}
\usepackage{mathrsfs}
\usepackage{tikz-cd}
\usepackage{multicol}
\usepackage[title]{appendix}
\usetikzlibrary{quotes,angles}

\newcommand{\C}{\mathbb{C}}

\newcommand{\A}{\alpha}
\newcommand{\B}{\beta}

\newtheorem{thm}{Theorem}[section]

\theoremstyle{definition}

\theoremstyle{remark}

\title{Classification of Low-dimensional Complex Triassociative Algebras}
\author{Erik Mainellis}
\date{}

\begin{document}

\maketitle

\begin{abstract}
    The paper concerns associative trialgebras, also known as triassociative algebras, which were first studied by Loday and Ronco in 2001. These generalize Loday's associative dialgebras (diassociative algebras) and are characterized by 3 operations and 11 identities. The paper details the classification of 1-dimensional and 2-dimensional triassociative algebras over a complex vector space.
\end{abstract}

\section{Introduction}
In 2001, Loday and Ronco showed that the family of chain modules over the standard simplices can be equipped with an operad structure (see \cite{loday tri}). Algebras over this operad are called associative trialgebras, or \textit{triassociative algebras}, and are characterized by 3 operations and 11 identities. The operad is Koszul and dual to the operad that corresponds to \textit{dendriform trialgebras} (see \cite{loday tri} and \cite{loday trialgebras}). Triassociative algebras generalize Loday's associative dialgebras, or \textit{diassociative algebras}, which were introduced in \cite{loday dialgebras}. Both diassociative and triassociative algebras have connections with algebraic topology, among other fields, and a complete classification of complex diassociative algebras has been obtained up to dimension 3 (see \cite{basri 3}). Furthermore, nilpotent diassociative algebras have been classified up to dimension 4 (see \cite{basri}). The cohomology of dialgebras was developed in \cite{loday dialgebras}, and that of triassociative algebras was studied in \cite{yau}. The objective of the present paper is to classify complex triassociative algebras up to dimension 2.

\section{Preliminaries}
Explicitly, a diassociative algebra $(D,\vdash,\dashv)$ consists of a vector space $D$ equipped with two bilinear operations $\vdash,\dashv:D\times D\xrightarrow{} D$ that satisfy
\begin{align*}
    (x\vdash y)\vdash z = x\vdash(y\vdash z) &~~~~~ \text{A1} & (x\dashv y)\dashv z = x\dashv(y\dashv z) &~~~~~ \text{A2} \\
    (x\dashv y)\vdash z = x\vdash (y\vdash z) &~~~~~ \text{D1} &(x\dashv y)\dashv z = x\dashv (y\vdash z) &~~~~~ \text{D2} \\
    (x\vdash y)\dashv z = x\vdash (y\dashv z) &~~~~~ \text{S1} & \text{self}
\end{align*}
for all $x,y,z\in D$. Note that there is a vertical symmetry between the two columns that reflects the order of operations and swaps $\vdash$ and $\dashv$. In this sense, S1 is self-symmetric. We also note that A1 and A2 yield associative algebras $(D,\vdash)$ and $(D,\dashv)$ respectively. The triassociative axioms build on these ones and introduce a third operation. A triassociative algebra $(T,\vdash,\dashv,\perp)$ is a vector space $T$ equipped with three bilinear products $\vdash,\dashv,\perp:T\times T\xrightarrow{} T$ such that $(T,\vdash,\dashv)$ is a diassociative algebra and
\begin{align*}
(x\perp y)\vdash z = x\vdash(y\vdash z) &~~~~~ \text{T1} & (x\dashv y)\dashv z = x\dashv(y\perp z) &~~~~~ \text{T2}\\ 
(x\vdash y)\perp z = x\vdash(y\perp z) &~~~~~\text{T3} &(x\perp y)\dashv z = x\perp (y\dashv z) &~~~~~ \text{T4} \\
(x\dashv y)\perp z = x\perp(y\vdash z) & ~~~~~\text{S2} & \text{self} \\
(x\perp y)\perp z = x\perp(y\perp z) & ~~~~~\text{A3} & \text{self}
\end{align*}
for all $x,y,z\in T$. We note that there is again a symmetry between the columns that reflects the order of operations and swaps $\vdash$ and $\dashv$. Moreover, $(T,\perp)$ forms an associative algebra.

The proofs in this paper follow a similar procedure to that in \cite{basri 3}. Therein, the authors classify diassociative algebras by fixing associative structures on one of the operations and deducing all possible diassociative structures from there. We thus provide the complete list of low-dimensional complex associative algebras here for their use in triassociative classification. The list in Theorem \ref{as 2} can be found in \cite{basri 3}, but we make a slight correction to $As_2^6$. Theorem \ref{as 1} is easily verified via a change of basis. For the rest of the paper, we use $x$ and $y$ to denote the basis elements of our algebras. In each class of algebra, all non-specified multiplications of basis elements are assumed to be zero.

\begin{thm}\label{as 1}
Let $A$ be a 1-dimensional complex associative algebra. Then $A$ is either abelian or isomorphic to an algebra with multiplication $xx = x$.
\end{thm}

\begin{thm}\label{as 2}
Let $A$ be a 2-dimensional complex associative algebra. Then $A$ is isomorphic to one of the following pairwise non-isomorphic associative algebras:
\begin{enumerate}
    \item[] $As_2^1$ : abelian;
    \item[] $As_2^2$ : $xx = x$;
    \item[] $As_2^3$ : $xx = y$;
    \item[] $As_2^4$ : $xx = x$, $xy = y$;
    \item[] $As_2^5$ : $xx = x$, $yx = y$;
    \item[] $As_2^6$ : $xx = x$, $xy = y$, $yx = y$;
    \item[] $As_2^7$ : $xx = x$, $yy = y$.
\end{enumerate}
\end{thm}

\section{Classification of 1-dimensional complex triassociative algebras}

\begin{thm}
Let $T$ be a 1-dimensional complex triassociative algebra. Then $T$ is isomorphic to one of the following pairwise non-isomorphic triassociative algebras:
\begin{enumerate}
    \item[] $Trias_1^1$ : abelian;
    \item[] $Trias_1^2$ : $x\perp x = x$;
    \item[] $Trias_1^3$ : $x\vdash x = x\dashv x = x\perp x = x$.
\end{enumerate}
\end{thm}

\begin{proof}
Let $\{x\}$ be a basis for $T$. We begin by choosing an associative structure on $(T,\vdash)$, which must either be abelian or such that $x\vdash x = x$. We first assume that it is abelian and denote $x\dashv x = \alpha x$ and $x\perp x = \beta x$ for some $\alpha,\beta\in \C$. The restrictions on these coefficients can be found by plugging $x$'s into the axioms of triassociative algebras. We first note that axioms A1, D1, and S1 yield only $0=0$, as both sides of each equation have a $\vdash$ multiplication. Computing both sides of $(x\dashv x)\dashv x = x\dashv (x\vdash x)$ (D2), however, yields $\alpha^2x = 0$, which means that $\alpha$ must be zero, and so $x\dashv x = 0$. Thus, axioms A2, T1, T2, T3, T4, and S2 are all trivial since both sides of each equation have either a $\vdash$ or $\dashv$. Axiom A3, however, yields $\beta^2x = \beta^2x$, and so no restrictions are placed on $\beta$. If $\beta=0$, we obtain $T\cong Trias_1^1$. If $\beta\neq 0$, a change of basis yields $Trias_1^2$.

Now assume that $x\vdash x = x$ and let $x\dashv x = \alpha x$ and $x\perp x = \beta x$ once more. Axioms S1 and A1 yield only trivial equalities. Let us consider D1; computing both sides of $(x\dashv x)\vdash x = x\vdash(x\vdash x)$ yields $\alpha x = x$, and we obtain $\alpha=1$. Computing $(x\dashv x)\dashv x = x\dashv (x\perp x)$ yields $x=\beta x$, and so $\beta =1$. All other axioms give nothing new, and so $T$ must be isomorphic to $Trias_1^3$.
\end{proof}

\section{Classification of 2-dimensional complex triassociative algebras}

\begin{thm}\label{tri 2 dim}
Let $T$ be a 2-dimensional complex triassociative algebra. Then $T$ is isomorphic to one of the following pairwise non-isomorphic triassociative algebras:
\begin{enumerate}
    \item[] $Trias_2^1$ : abelian;
    \item[] $Trias_2^2$ : $y\perp y = x$;
    \item[] $Trias_2^3$ : $y\perp y = \alpha x + y$ where $\alpha\in \C$;
    \item[] $Trias_2^4$ : $y\perp x = x$, $y\perp y = \alpha y$ where $\alpha\in \C\setminus \{0\}$;
    \item[] $Trias_2^5$ : $x\perp y =x$, $y\perp y = \alpha y$ where $\alpha\in \C\setminus \{0\}$;
    \item[] $Trias_2^6$ : $x\perp y = y\perp x = x$, $y\perp y = \alpha x + y$ where $\alpha\in \C$;
    \item[] $Trias_2^7$ : $x\perp x = x$, $y\perp y = \alpha y$ where $\alpha\in \C$;
    \item[] $Trias_2^8$ : $x\perp x = x$, $y\perp x = y$;
    \item[] $Trias_2^9$ : $x\perp x = x$, $x\perp y = y$, $y\perp x = \alpha x$, $y\perp y = \alpha y$ where $\alpha\in \C$;
    \item[] $Trias_2^{10}$ : $x\perp x = x$, $x \perp y = y$, $y\perp x = y$, $y\perp y = \alpha x + \beta y$ where $\alpha,\beta\in \C$;
    \item[] $Trias_2^{11}$ : $x\perp x = x$, $x\perp y = y\perp x = \alpha x$, $y\perp y = \alpha y$ where $\alpha\in \C\setminus\{0\}$;
    \item[] $Trias_2^{12}$ : $x\perp x = x$, $x\perp y = \alpha x$, $y\perp x = y$, $y\perp y = \alpha y$ where $\alpha\in \C\setminus \{0\}$;
    \item[] $Trias_2^{13}$ : $x\perp x = \alpha x + y$ where $\alpha\in \C$;
    \item[] $Trias_2^{14}$ : $x\perp x = \alpha_1 x + \beta y$, $x\perp y = y\perp x = \gamma y$, $y\perp y = \alpha_2 y$ where \[\gamma = \frac{\alpha_1 \pm \sqrt{\alpha_1^2 + 4\beta\alpha_2}}{2}\] and $\alpha_1,\alpha_2\in \C$, $\beta\in \C\setminus\{0\}$;
    \item[] $Trias_2^{15}$ : $x\perp x = \alpha x + y$, $x\perp y = y\perp x = \beta x$, $y\perp y= \beta y$ where $\alpha\in \C$ and $\beta\in \C\setminus\{0\}$;
    \item[] $Trias_2^{16}$ : $x\perp x = \alpha x + \beta_1 y$, $x\perp y = y\perp x = \gamma_1 x + \gamma_2 y$, $y\perp y = \beta_2 x$ where \[\gamma_1 = \sqrt[3]{-\frac{\beta_1\beta_2^2}{2} + \sqrt{\left(-\frac{\beta_1\beta_2^2}{2}\right)^2 + \left(-\frac{\alpha\beta_2}{3}\right)^3}} + \sqrt[3]{-\frac{\beta_1\beta_2^2}{2} - \sqrt{\left(-\frac{\beta_1\beta_2^2}{2}\right)^2 + \left(-\frac{\alpha\beta_2}{3}\right)^3}},\] \[\gamma_2 = \frac{\alpha\pm \sqrt{\alpha^2 - 4\beta_1\gamma_1}}{2},\] and $\alpha\in \C$, $\beta_1,\beta_2\in \C\setminus\{0\}$;
    \item[] $Trias_2^{17}$ : $x\perp x = \alpha x + \beta_1 y$, $x\perp y = y\perp x = \gamma_1 x + \gamma_2 y$, $\beta_2 x + \beta_3 y$ where \[\gamma_2 = \sqrt[3]{\left(-\frac{b^3}{27} + \frac{bc}{6} - \frac{d}{2}\right) + \sqrt{\left(-\frac{b^3}{27} + \frac{bc}{6} - \frac{d}{2}\right)^2 + \left(\frac{c}{3} - \frac{b^2}{9}\right)^3}}~~~~~~~~~~~~~~~~~~~\] \[+ \sqrt[3]{\left(-\frac{b^3}{27} + \frac{bc}{6} - \frac{d}{2}\right) - \sqrt{\left(-\frac{b^3}{27} + \frac{bc}{6} - \frac{d}{2}\right)^2 + \left(\frac{c}{3} - \frac{b^2}{9}\right)^3}} - \frac{b}{3},\] \[\gamma_1 = \frac{\beta_3\pm \sqrt{\beta_3^2 +4(\gamma_2\beta_2 + \alpha\beta_2)}}{2},\] $b=-\alpha$, $c=-\beta_1\beta_3$, $d=\beta_1^2\beta_2$, and $\alpha\in \C$, $\beta_1,\beta_2,\beta_3\in\C\setminus\{0\}$;
    \item[] $Trias_2^{18}$ : $y\dashv y = y$, $x\perp x = \alpha x$ where $\alpha\in \C$;
    \item[] $Trias_2^{19}$ : $y\dashv y = x$, $y\perp y = \alpha x$ where $\alpha\in \C$;
    \item[] $Trias_2^{20}$ : $x\dashv x = y$, $x\perp x = \alpha y$ where $\alpha\in\C$;
    \item[] $Trias_2^{21}$ : $y\dashv x = x$;
    \item[] $Trias_2^{22}$ : $x\dashv y = x$, $y\dashv x = \alpha x$ where $\alpha\in\C$;
    \item[] $Trias_2^{23}$ : $x\dashv y = x$, $y\perp x = \alpha x$ where $\alpha\in\C$;
    \item[] $Trias_2^{24}$ : $x\dashv x = \gamma_1 x + \alpha_1 y$, $x\dashv y = y\dashv x = -\gamma_2 x - \gamma_1 y$, $y\dashv y = \alpha_2 x + \gamma_2 y$, $x\perp x = \delta_1x + \frac{\alpha_1\beta}{\alpha_2}y$, $x\perp y = y\perp x = -\delta_2 x - \delta_1 y$, $y\perp y = \beta x + \delta_2 y$ where \[\gamma_1 = \sqrt[3]{\alpha_1^2\alpha_2}, ~\gamma_2 = \sqrt[3]{\alpha_1\alpha_2^2},~ \delta_1 = \sqrt[3]{\frac{\alpha_1^2}{\alpha_2^2}}\beta, ~ \delta_2 = \sqrt[3]{\frac{\alpha_1}{\alpha_2}}\beta\] and $\alpha_1,\alpha_2\in\C\setminus\{0\}$, $\beta\in \C$;
    \item[] $Trias_2^{25}$ : $x\vdash x = x$, $x\dashv x = x$, $y\dashv x = y$, $x\perp x = x$, $y\perp x = y$;
    \item[] $Trias_2^{26}$ : $x\vdash x = x$, $x\dashv x = x$, $y\dashv x = y$, $x\perp x = x$;
    \item[] $Trias_2^{27}$ : $x\vdash x = x$, $x\dashv x = x$, $x\perp x = x$, $y\perp y = \alpha y$ where $\alpha\in\C$;
    \item[] $Trias_2^{28}$ : $x\vdash x = y$, $x\dashv x = \alpha y$, $x\perp x = \beta y$ where $\alpha,\beta\in\C$;
    \item[] $Trias_2^{29}$ : $x\vdash x = x$, $x\vdash y = y$, $x\dashv x = x$, $x\perp x = x$;
    \item[] $Trias_2^{30}$ : $x\vdash x = x$, $x\vdash y = y$, $x\dashv x =x$, $x\perp x = x$, $x\perp y = y$;
    \item[] $Trias_2^{31}$ : $x\vdash x = x$, $x\vdash y = y$, $x\dashv x = x$, $y\dashv x = y$, $x\perp x = x$, $y\perp y = \alpha y$ where $\alpha\in\C$;
    \item[] $Trias_2^{32}$ : $x\vdash x = x$, $x\vdash y = y$, $x\dashv x = x$, $y\dashv x = y$, $x\perp x = x$, $y\perp x = y$;
    \item[] $Trias_2^{33}$ : $x\vdash x = x$, $x\vdash y = y$, $x\dashv x = x$, $y\dashv x = y$, $x\perp x = x$, $x\perp y = y$;
    \item[] $Trias_2^{34}$ : $x\vdash x = x$, $x\vdash y = y$, $x\dashv x = x$, $y\dashv x = y$, $x\perp x = x$, $x\perp y = y\perp x = y$, $y\perp y = \alpha y$ where $\alpha\in \C$;
    \item[] $Trias_2^{35}$ : $x\vdash x = x$, $x\vdash y = y$, $x\dashv x = x$, $y\dashv x = y$, $x\perp x = x+\alpha y$ where $\alpha\in\C\setminus\{0\}$;
    \item[] $Trias_2^{36}$ : $x\vdash x = x$, $x\vdash y = y$, $x\dashv x = x$, $y\dashv x = y$, $x\perp x = x+\alpha y$, $x\perp y = y\perp x = \gamma y$, $y\perp y = \beta y$ where \[\gamma = \frac{1\pm \sqrt{1+4\alpha\beta}}{2}\] and $\alpha\in\C\setminus\{0\}$, $\beta\in \C$;
    \item[] $Trias_2^{37}$ : $x\vdash x = x$, $x\vdash y = y$, $x\dashv x = x$, $x\dashv y = y$, $x\perp x = x$, $x\perp y = y$;
    \item[] $Trias_2^{38}$ : $x\vdash x = x$, $x\vdash y = y$, $x\dashv x = x\perp x = x+\alpha y$ where $\alpha\in\C\setminus\{0\}$;
    \item[] $Trias_2^{39}$ : $x\vdash x = x$, $x\vdash y = y$, $x\dashv x = x+\alpha y$, $x\perp x = x$, $x\perp y = y$ where $\alpha\in \C\setminus\{0\}$;
    \item[] $Trias_2^{40}$ : $x\vdash x = x$, $y\vdash x = y$, $x\dashv x = x$, $y\dashv x = y$, $x\perp x = x$, $y\perp x = y$;
    \item[] $Trias_2^{41}$ : $x\vdash x = x$, $x\vdash y = y\vdash x = y$, $x\dashv x = x$, $x\dashv y = y\dashv x = y$, $x\perp x = x$, $x\perp y = y\perp x = y$;
    \item[] $Trias_2^{42}$ : $x\vdash x = x$, $y\vdash y = y$, $x\dashv x = x$, $y\dashv y = y$, $x\perp x = x$, $y\perp y = y$.
\end{enumerate}
\end{thm}

\begin{proof}
Given our triassociative algebra $(T,\vdash,\dashv,\perp)$ with basis $\{x,y\}$, the proof proceeds by considering the cases as $(T,\vdash)$ ranges over $As_2^i$. Here, we will detail the case of $(T,\vdash) = As_2^2$, as it provides a reasonable demonstration of the procedure and returns several isomorphism classes, but is not excessively long. The other cases follow by the same logic. We thus set $x\vdash x = x$ and denote
\begin{align*}
    x\dashv x = \A_1x + \A_2y && x\dashv y = \A_3x + \A_4y \\ y\dashv x = \A_5x + \A_6y && y\dashv y = \A_7x + \A_8y \\ x\perp x = \B_1x + \B_2y && x\perp y = \B_3x + \B_4y \\ y\perp x = \B_5x + \B_6y && y\perp y = \B_7x + \B_8y
\end{align*} \noindent for some $\alpha_i,\beta_i\in \C$. We then proceed to plug all orderings of our two basis elements into the three positions in each identity. For the identity $(~\vdash~)\dashv~ = ~\vdash(~\dashv~)$, for example, this consists of computing the equalities
\begin{align*} (x\vdash x)\dashv x = x\vdash (x\dashv x) && (y\vdash x)\dashv x = y\vdash (x\dashv x) \\ (x\vdash y)\dashv x = x\vdash (y\dashv x) && (x\vdash x)\dashv y = x\vdash (x\dashv y) \\ (y\vdash y)\dashv x = y\vdash (y\dashv x) && (y\vdash x)\dashv y = y\vdash (x\dashv y) \\ (x\vdash y)\dashv y = x\vdash (y\dashv y) && (y\vdash y)\dashv y = y\vdash (y\dashv y)
\end{align*} with orderings $xxx$, $yxx$, $xyx$, $xxy$, $yyx$, $yxy$, $xyy$, $yyy$ of variables. This identity is a particularly good choice for starting the $(T,\vdash) = As_2^2$ case, since it yields $\A_2 = \A_4 = \A_5 = \A_7 =0$. Our multiplication structure on $T$ can thus be rewritten as follows. \begin{align*}
    x\dashv x = \A_1x && x\dashv y = \A_3x \\ y\dashv x = \A_6y && y\dashv y = \A_8y \\ x\perp x = \B_1x + \B_2y && x\perp y = \B_3x + \B_4y \\ y\perp x = \B_5x + \B_6y && y\perp y = \B_7x + \B_8y
\end{align*} Now consider the identity $(~\dashv ~)\dashv~= ~\dashv(~\dashv ~)$. Under our current multiplication, it yields $\A_6^2 = \A_1\A_6$, $\A_3^2 = \A_3\A_8$, and $\A_1\A_3 = \A_3\A_6 =\A_6\A_8$. From $(~\dashv ~)\vdash~= ~\vdash(~\vdash ~)$, we obtain $\A_1 = 1$ and $\A_3 = 0$. Combining these two new collections of relations yields $\A_6^2 = \A_6$ and $\A_6\A_8 = 0$. Applying these restrictions to $(~\dashv ~)\dashv~= ~\dashv(~\vdash ~)$, we get one new fact, that $\A_8^2 = 0$, which implies that $\A_8=0$. Our multiplications can now be rewritten as \begin{align*}
    x\dashv x = x && x\dashv y = 0 \\ y\dashv x = \A_6y && y\dashv y = 0 \\ x\perp x = \B_1x + \B_2y && x\perp y = \B_3x + \B_4y \\ y\perp x = \B_5x + \B_6y && y\perp y = \B_7x + \B_8y
\end{align*} subject to the constraint $\A_6^2 = \A_6$. Before moving on to our relations involving $\perp$, we note that the associativity of $\vdash$ yields nothing new in the context of our current case. We next consider the identity $(~\dashv ~)\dashv~= ~\dashv(~\perp ~)$, which returns $\B_1=1$ and $\B_3 = \B_5 = \B_7 = 0$. Continuing through the triassociative axioms, we obtain $\B_2 = \B_4 = 0$, $\B_6^2 = \B_6 = \A_6\B_6$, and $\A_6\B_8 = 0$. Thus, our multiplications can be rewritten as \begin{align*}
    x\dashv x = x && x\dashv y = 0 \\ y\dashv x = \A_6y && y\dashv y = 0 \\ x\perp x = x && x\perp y = 0 \\ y\perp x = \B_6y && y\perp y = \B_8y
\end{align*} subject to $\A_6^2=\A_6$, $\B_6^2= \B_6 = \A_6\beta_6$, and $\A_6\B_8 = 0$. We now make deductions from these relations on the complex numbers. First, we have $\A_6(\A_6 - 1) = 0$, and so $\A_6$ must equal either 0 or 1. Supposing the former, we obtain $\beta_6 = 0$, and any relations restricting $\B_8$ vanish. In this case, our algebra is isomorphic to $Trias_2^{27}$. Supposing $\A_6 = 1$, we obtain $\B_8 = 0$ and that $\B_6$ is equal to either 0 or 1. If $\B_6=0$, then our algebra is isomorphic to $Trias_2^{26}$. If $\B_6 = 1$, then it is isomorphic to $Trias_2^{25}$.
\end{proof}

We remark that classes $Trias_2^{16}$ and $Trias_2^{17}$ make use of the cubic root formula, which arises based on the relations between coefficients. Moreover, the isomorphism classes of 2-dimensional complex triassociative algebras can be collected based on their associative $\vdash$ structure, and we conclude with a table that arranges them in this manner.

\begin{center}
    \begin{tabular}{c|c}
        Associative Algebra $(T,\vdash)$ & Corresponding Triassociative Algebras $(T,\vdash,\dashv,\perp)$ \\ \hline \\
        $As_2^1$ & $Trias_2^1$, $Trias_2^2$, $Trias_2^3$, \dots, $Trias_2^{24}$ \\[10pt] $As_2^2$ & $Trias_2^{25}$, $Trias_2^{26}$, $Trias_2^{27}$ \\[10pt] $As_2^3$ & $Trias_2^{28}$ \\[10pt] $As_2^4$ & $Trias_2^{29}$, $Trias_2^{30}$, $Trias_2^{31}$, \dots, $Trias_2^{39}$ \\[10pt] $As_2^5$ & $Trias_2^{40}$ \\[10pt] $As_2^6$ & $Trias_2^{41}$ \\[10pt] $As_2^7$ & $Trias_2^{42}$
    \end{tabular}
\end{center}


\bigskip

\begin{thebibliography}{}

\bibitem{basri 3} Basri, W.; Rakhimov, I.; Rikhsiboev, I. ``Classification of 3-Dimensional Complex Diassociative Algebras." \textit{Malaysian Journal of Mathematical Sciences}, Vol. 4, No. 2 (2010).

\bibitem{basri} Basri, W.; Rakhimov, I.; Rikhsiboev, I. ``Four-Dimensional Nilpotent Diassociative Algebras.'' \textit{Journal of Generalized Lie Theory and Applications}, Vol. 9, No. 1 (2015).

\bibitem{loday dialgebras} Loday, J.-L. ``Dialgebras" in \textit{Dialgebras and related operads}, pp. 7-66. Lecture Notes in Mathematics, Vol. 1763. Springer-Verlag Berlin Heidelberg (2001).

\bibitem{loday tri} Loday, J.-L.; Ronco, M. ``A duality between standard simplices and Stasheff polytopes" (2001).	arXiv:math/0102089

\bibitem{loday trialgebras} Loday, J.-L.; Ronco, M. ``Trialgebras and families of polytopes" (2002). arXiv:math/0205043

\bibitem{yau} Yau, D. ``(Co)homology of triassociative algebras." \textit{International Journal of Mathematics and Mathematical Sciences}, Vol. 2006, No. 9 (2006).

\end{thebibliography}
\end{document}